\documentclass{amsart}

\usepackage{amsmath,amssymb,mathrsfs}
\usepackage{euscript}
\usepackage{color} 
\usepackage{caption}

\usepackage[bookmarks=false]{hyperref}
\hypersetup{
 colorlinks=true, 	 
 linktoc=all,     	 
 linkcolor=blue,  	 
 citecolor=blue,    
 filecolor=magenta,  
 urlcolor=cyan       
}
\usepackage{tikz} 		
\usepackage{tikz-cd}
\usetikzlibrary{matrix,arrows,decorations.pathmorphing,intersections}

\newtheorem{thm}{Theorem}[section]

\newtheorem{lemma}[thm]{Lemma}
\newtheorem{prop}[thm]{Proposition}
\newtheorem{cor}[thm]{Corollary}

\theoremstyle{definition}
\newtheorem{defn}[thm]{Definition}

\newtheorem{ex}[thm]{Example}

\newcommand{\Z}{{\mathbb{Z}}}

\newcommand{\C}{{\mathbb{C}}}

\newcommand{\Out}{{\rm Out}}
\newcommand{\Inn}{{\rm Inn}}
\newcommand{\Aut}{{\rm Aut}}
\newcommand{\IA}{{\rm IA}}
\newcommand{\Flag}{{{\rm Flag}}}

\newcommand{\GL}{{{\rm GL}}}
\newcommand{\SL}{{{\rm SL}}}
\newcommand{\PSL}{{{\rm PSL}}}

\newcommand{\colim}{{{\rm colim}}}

\usepackage{enumitem}
\setlist[itemize]{leftmargin=0.3in}
\setlist[enumerate]{leftmargin=0.3in}

\title[Polyhedral products, flag complexes and monodromy]{Polyhedral products, flag complexes and monodromy representations}

\author{Mentor Stafa}
\address{ Departement Mathematik, ETH Z\"urich, 8092, Switzerland.}
\address{Current: Department of Mathematics, Tulane University, New Orleans, LA, 70118}
\email{ mstafa@tulane.edu}

\date{}
\subjclass[2010]{Primary: 55U10, 58K10, 20F65, 14F45.}
\thanks{Supported by a Swiss Government Excellence Scholarship (ESKAS No. 2015.0182), 2015-2016}
\thanks{Previous addresses: Departement Mathematik, ETH Z\"urich, CH-8092, Switzerland \\ and Indiana University Purdue University, Mathematical Sciences, Indianapolis, IN   }
\keywords{polyhedral product, monodromy representation, graph group, CT group, Magma}

\begin{document}

\begin{abstract}
This article presents a machinery based on polyhedral products
that produces faithful representations of graph products of finite groups and 
direct products of finite groups into automorphisms of free groups $\Aut(F_n)$
and outer automorphisms of free groups $\Out(F_n)$, respectively, as well as 
faithful representations of products of finite groups into 
the linear groups $\SL(n,\Z)$ and $\GL(n,\Z)$.
These faithful representations are realized as monodromy representations.
\end{abstract}

\maketitle

\tableofcontents

\section{Introduction}

Studying the topology of a fibration sequence frequently involves  
the  monodromy action, which is the action of the fundamental group $\pi$ of 
the base space $B$ on the fibre $F$. 
When using a spectral sequence one may need to consider the homology of the base with
coefficients in the homology of the fibre regarded as an $R\pi$-module, where $R\pi$ is 
the group ring.
Here the monodromy representation for a fibration $p:E\to B$ with
fibre $F$ will mean the representation 
$\rho: \pi_1(B) \to \Out(H_1(F)).$
The goal of this paper is to study polyhedral products in connection with 
monodromy representations for certain fibrations that arise naturally in 
the field of \textit{toric topology}. 
The problem of explicitly describing such representations was studied 
by the author in \cite{stafa.monodromy}.

Numerous results on graph products of groups and polyhedral products demonstrate that the
underlying simplicial complex $K$ plays an important role in their study.
For example Droms \cite{droms1987isomorphisms} proved that two graph products of groups
are isomorphic if and only if the graphs are isomorphic.
Servatius, Droms, and Servatius~\cite{servatius1989surface}
determine the simplicial complexes for which the commutator subgroup of a 
right-angled Artin groups is free. 
Moreover, if $K$ is chosen carefully, one obtains classifying spaces 
for various important families of discrete groups, including right-angled Artin 
and Coxeter groups from geometric group theory \cite{stafa.monodromy,davis.okun}.
In another application Grbi\'c, Panov, Theriault and Wu \cite{wu.grbic.panov}
give conditions on the 1-skeleton of a flag complex $K$ that determine when the face ring of $K$ 
is a Golod ring, or equivalently the corresponding moment-angle complex
has the homotopy type of a wedge of spheres.
Recently Panov and Veryovkin \cite{panov2016polyhedral} studied
polyhedral products that have the homotopy type of classifying spaces of right-angled 
Artin groups and right-angled Coxeter groups.

\

In the present article we study further properties of the monodromy representations
associated to the (homotopy) fibration sequences
\begin{equation}\label{eqn: D-S-fibration INTRO1}
(\underline{EG},\underline{G})^K \to (\underline{BG},\underline{1})^K 
				\to \prod_{i=1}^n BG_i.
\end{equation}
Each space in \eqref{eqn: D-S-fibration INTRO1} is a polyhedral product, 
depending on a simplicial complex $K$, together with a sequence of 
finite groups $\underline{G}:=\{G_1,\dots,G_n\}$, their classifying spaces
$\underline{BG}:=\{BG_1,\dots,BG_n\}$, and corresponding universal 
covers  $\underline{EG}:=\{EG_1,\dots,EG_n\}$,  
see Definition \ref{defn: polyhedral product}.

We give explicit descriptions of monodromy representations for simplicial
complexes $K$ with more than two vertices, which were described 
geometrically in \cite{stafa.monodromy}.
To do this we generalize and use some results of
Panov and Veryovkin \cite{panov2016polyhedral}.
We give applications, in particular to spaces of commuting elements in 
\textit{commutative transitive (CT) finite groups}, where commutativity
is a transitive relation, studied in a 
celebrated paper of M. Suzuki \cite{suzuki1957nonexistence},
and an application to a problem related to the Feit-Thompson theorem,
which states that all groups of odd order are solvable.
Finally, we give a couple of examples, which can be generalized 
using \verb|Magma|~\cite{magma}.

\subsection*{Main results} 

For given finite discrete groups $G_1,\dots, G_n$, we use polyhedral 
products to construct monodromy representations
$ \Phi: G_1 \times \cdots \times G_n \to \Out(F_N) $ 
into outer automorphism groups of free groups. In particular, we obtain explicit 
faithful representations of graph products of
finite groups into automorphism groups of free groups, 
and faithful representations of their direct products into 
linear groups $\SL(k,\Z)$ or $\GL(k,\Z)$. 
This article presents a machinery based on polyhedral products to achieve this.
The first result is the following theorem.

\begin{thm}\label{thm: theorem 1 intro}
Let $G_1,\dots, G_n$ be finite groups and $K$ a simplicial complex
with $n$ vertices with 1-skeleton $K^1$ a chordal graph. 
Then there are faithful representations
$$\Theta_K: \prod_{K^1} G_i \to \Aut(F_{\rho_K})$$
and faithful monodromy representations
$$\Phi_K: G_1\times \cdots \times G_n \to \Out(F_{\rho_K}),$$
where $\rho_K$ is the rank of the fundamental group of the fibre in equation
(\ref{eqn: D-S-fibration INTRO1}).
\end{thm}

The case when the groups $G_1,\dots,G_n$ are abelian the representations 
can be described explicitely and convenient models of polyhedral products can
then be used to show that the corresponding monodromy representations 
obtained for non-abelian finite groups are also faithful.

\begin{thm}\label{thm: theorem 2 intro}
Let $G_1,\dots, G_n$ be finite abelian groups. Then the faithful
monodromy representation $\Phi_K$ induces a faithful representation
$$\Phi_K: G_1\times\cdots \times G_n \to \SL({\rho_K},\Z).$$
If $G_1,\dots,G_n$ are non-abelian then $\Phi_K$ maps into $\GL({\rho_K},\Z)$.
\end{thm}

Let $E(2,G)\subseteq EG$ and $B(2,G) \subseteq BG$ be the spaces
defined in $\S$\ref{sec: B(2,G) and E(2,G)} that classify commuting elements in 
a group $G$. In particular, we use polyhedral products to study the class of
finite transitively commutative (CT) groups, a class of groups where commtutativity 
is transitive. The following theorem is then an application of 
polyhedral products to group theory.

\begin{thm}\label{thm: top. equiv. form CT groups INTRO}
Finite CT groups with trivial center are solvable if and only if the induced map
$H_1(E(2,G);\Z) \to H_1(B(2,G);\Z)$ is not surjective.
\end{thm}

This theorem is motivated from a result of Adem, Cohen and Torres-Giese 
\cite{fredb2g}, which states an equivalent topological condition to the the 
Feit-Thompson theorem, namely that the theorem is true if and
only if the map $H_1(E(2,G);\Z) \to H_1(B(2,G);\Z)$ is not surjective.

\subsection*{Structure of paper} In Section \ref{sec: polyhedral products and fibrations} we define polyhedral products and the fibration sequences we work with. We study the commutator subgroup of graph products of groups in Section \ref{sec: commutator subgp}, and find bases for the free groups, wich will be used in Section \ref{section: e.g.} to provide examples. We prove Theorems \ref{thm: theorem 1 intro} and \ref{thm: theorem 2 intro} in Section \ref{sec: graph products abelian gps}. Finally applications are given in Sections \ref{sec: induced map in homology} and \ref{sec: B(2,G) and E(2,G)}, where we also prove Theorem \ref{thm: top. equiv. form CT groups INTRO}.

\subsection*{Acknowledgments} The author thanks Alina Vdovina 
(University of Newcastle, UK), who was visiting the Institute 
for Mathematical Research at ETH Z\"urich, during the spring 
semester 2016, for our numerous conversations 
on the topic, for suggesting me to use \verb|Magma| 
for some computations, and for providing the first codes.
The author also thanks the anonymous referee(s) for the valuable
comments that helped improve the exposition of the paper.

\section{Polyhedral products and related fibrations}\label{sec: polyhedral products and fibrations}

{Moment-angle complexes}, originally invented by Davis and Januszckiewicz \cite{davis.januszckiewicz},
 appeared also in the work of Buchstaber and Panov \cite{buchstaber2002torus}
in the context of \textit{toric topology}. Polyhedral products are a generalization of moment-angle complexes and were introduced and popularized by work of Bahri,  Bendersky,  Cohen and  Gitler \cite{cohen.macs}, 
and are the main objects of study in toric topology; see the more recent monograph 
by Buchstaber and Panov \cite{bp2015}. 

\begin{defn}\label{defn: polyhedral product}
Let $(\underline{X},\underline{A} )$ denote a sequence of 
pointed $CW$-pairs $\{(X_i,A_i)\}_{i=1}^n$ 
and $[n]$ denote the sequence of integers $\{1,2,\dots,n\}$.

\begin{itemize}
\item A \textbf{simplex} $\sigma$ is given by an increasing 
		sequence of integers $\sigma = \{1 \leq i_1 < \cdots < i_q \leq n\}
		\subseteq [n]$.
		A {\bf simplicial complex} $K$ is a collection of simplices such that
		if $\tau\subset \sigma$ and $\sigma \in K,$ then $\tau \in K.$
		In particular $\emptyset \in K$. Geometrically $K$ is a subcomplex of the the $(n-1)$-simplex $\Delta[n-1]$. 
		If $K$ has only 0- and 1-simplices we call it a simplicial graph.
		
\item The \textbf{polyhedral product}  $(\underline{X},\underline{A})^K$ 
		is the subspace of the product $X_1 \times \cdots \times X_n$ 
		given by the colimit
$$
(\underline{X},\underline{A})^K := 
		\underset{{\sigma \in K}}{\colim}\, \mathcal D(\sigma)
			= \bigcup_{\sigma \in K} \mathcal D(\sigma) \subseteq \prod_{i=1}^n X_i, 
$$
		where $\mathcal D(\sigma)=\{(x_1,\dots,x_n)\in \prod_{i=1}^n X_i | 
			x_i\in A_i \text{ if } i \notin \sigma\}$,
		the maps are the inclusions, and the topology is the 
		subspace topology of the product. Another standard notation for 
		polyhedral products is $Z_K(\underline{X},\underline{A})$.  
		Sometimes polyhedral products are called \textbf{$K$-powers}.
		Since $\emptyset$ is in any $K$, we have 
		$\prod_{i=1}^n A_i \subset (\underline{X},\underline{A})^K 
		\subset \prod_{i=1}^n X_i.$

\item If the pairs $(X,A)$ are $(D^2,S^1)$ or $(D^1,S^0)$, then the polyhedral
		products are called \textbf{moment-angle complexes} and 
		\textbf{real moment-angle complexes}, respectively.

\item If all the pairs in the sequence $(X_1,A_1),\dots,(X_n,A_n)$ 
		are equal to $(X,A)$, then we
		omit the underline in the notation of the polyhedral product, and write simply
		$({X},{A})^K$.
\end{itemize}
\end{defn}

\begin{defn}\label{defn: many definitions}
Next we give some relevant definitions and notation:

\begin{itemize}
\item For a simplicial complex $K$, the complex $K^i$ denotes 
		the \textbf{$i$-skeleton} of $K$.

\item A simplicial complex $K$ is called a \textbf{flag complex} 
		if for any complete subgraph $\Gamma \subset K^1$, 
		it also contains the simplex spanned by these vertices.
		The structure of moment angle complexes (or polyhedral products in general)
		is better understood when $K$ is a flag complex \cite[\S 8.5]{bp2015}.
		
\item For a simplicial complex $K$, let $\Flag(K)$ 
		denote the \textbf{clique complex} of $K^1$,
		i.e. the simplicial complex whose simplices are 
		complete subgraphs of $K^1$. 
		For example a flag complex is the clique complex of its 1-skeleton.
		
\item A graph is called \textbf{chordal} if every cycle of length greater than three
		has an edge (called a \textit{chord}) connecting two nonconsecutive vertices.
		Chordal graphs are also called \textbf{triangulated} graphs.
		
\item For any group $G$ denote its abelianization by $\mathscr A(G):=G/[G,G]$, and the 
		abelianization map by ${\rm ab}_G: G\twoheadrightarrow \mathscr A(G)$.		
		
\item 	Let $G_1,\dots,G_n$  be a sequence of groups and $\Gamma$ 
		a simplicial graph on $[n]$. The \textbf{graph product} of 
		$G_1,\dots,G_n$ over $\Gamma$ is the quotient of their 
		free product by the normal closure of the relations 
		$R_{\Gamma}:=\{[g_i,g_i]: \{i,j\}\text{ is an edge in }\Gamma\}.$ 
		The group obtained this way will be called a 
		\textbf{graph group}, even though in the literature
		this name is sometimes used for right-angled Artin
		groups. We denote it by 
		$$\prod_\Gamma G_i := (G_1\ast \cdots \ast G_n)/\langle R_{\Gamma} \rangle.$$
		In this notation right-angled Artin	groups are graph products of the group $\Z$.
\end{itemize}
\end{defn}

\begin{ex}\label{example: polyhedral products}

\

\begin{enumerate}[label=\arabic*.]
\item Let $X$ be the unit interval $[0,1]$ and $A\subset [0,1]$ be the subset
		$\{0,1/2,1\}$. Let $K$ be the simplicial complex consisting of only
		two vertices $\{\{v_1\},\{v_2\}\}$. Then 
		$\mathcal D(\{v_1\})=X\times A \subset [0,1]^2$
		and $\mathcal D(\{v_2\})=A\times X \subset [0,1]^2$.
		Therefore, $(X,A)^K= \mathcal D(\{v_1\}) \cup \mathcal D(\{v_2\})
		= X\times A \cup A\times X$ is a graph inside the square $[0,1]^2$,
		homotopy equivalent to a wedge of 4 circles $\bigvee_4 S^1$.
		Similarly, we can choose $A$ to be any finite subset of the unit interval
		and we obtain similar graphs.

\item Let $(X,A)=(D^2,S^1)$. If $K$ is a the boundary of the $n$-simplex
		then the moment-angle complex 
		$(D^2,S^1)^K = \bigcup_{\sigma_i} \mathcal D(\sigma_i)$
		is homeomorphic to the sphere $S^{2n+1}=\partial D^{2(n+1)}$.
		It is also known \cite[Theorem 4.6]{wu.grbic.panov} that
		if $K$ is a flag complex, then $(D^2,S^1)^K$ has the homotopy
		of a wedge of spheres if and only if $K^1$ is chordal.
\item If $K$ is any simplicial complex on $n$ vertices, and 
		$\underline{\ast}=\{\ast_1,\dots,\ast_n\}$ is the sequence of basepoints
		then $\underline{X}$ then $(\underline{X},\underline{\ast})^{K^0}
		=\bigvee_{i=1}^n X_i.$ Therefore, in general
		$\bigvee_{i=1}^n X_i \subseteq (\underline{X},\underline{\ast})^{K}
		\subseteq \prod_{i=1}^n X_i$.
\end{enumerate}
\end{ex}

Let $G$ be a topological group with basepoint its identity element $\ast=1$, $BG$ be the 
classifying space of $G$, with $B1\simeq 1 = \ast$, and $EG$ be a (weakly) contractible space with a free action of $G$
such that the quotient map $EG \to BG$ is a principal $G$-bundle. 
G. Denham and A. Suciu \cite[Lemma 2.3.2]{denham} gave a natural fibration relating 
the polyhedral product for the pair $(BG,1)$ to the polyhedral product for the pair $(EG,G)$. 
That is, for a simplicial complex $K$ with $n$ vertices, the polyhedral product $(BG,1)^K$ 
fibres over the product $(BG)^n$ as follows
\begin{equation}\label{eqn: D-S-fibration 1}
(EG,G)^K \to (EG)^n \times_{G^n} (EG,G)^K \to (BG)^n ,
\end{equation}
where the total space is homotopy equivalent to $(BG,1)^K$. Note that the group $G$ acts 
coordinate--wise on the fibre $(EG,G)^K \subset (EG)^n$.

This fibration is a generalization of the Davis-Januszkiewicz space 
\cite{davis.januszckiewicz}, with topological group the circle $S^1$, 
given by the Borel construction
$$\mathcal{DJ}(K) = (ES^1)^n \times_{(S^1)^n} (ES^1,S^1)^K,$$
which describes a cellular realization of the Stanley-Reisner ring of $K$, 
in the sense that the cohomology ring of $\mathcal{DJ}(K)$ is precisely 
the Stanley-Reisner ring of $K$ defined as the quotient of the polynomial ring 
$R[K]=R[x_1,\dots,x_n]/I_K$ by the Stanley-Reisner ideal 
$I_K=\langle x_{i_1}\cdots x_{i_t}| \{i_1,\dots,i_t\} \neq K \rangle$.
A later result of V. Buchstaber and T. Panov \cite{buch.panov} showed that
$\mathcal{DJ}(K)\simeq (BS^1,1)^K$ and
the homotopy fibre of the natural inclusion $\mathcal{DJ}(K)  \to (\C P^\infty)^n$ is 
equivalent to the polyhedral product $(ES^1,S^1)^K$.

If $G_1,\dots,G_n$ is a sequence of topological groups, then for any 
simplicial complex $K$ with $n$ vertices, the fibration sequence 
(\ref{eqn: D-S-fibration 1}) can be generalized to obtain 
\begin{equation}\label{eqn: D-S-fibration}
(\underline{EG},\underline{G})^K \to (\underline{BG},\underline{1})^K \to \prod_{i=1}^n BG_i.
\end{equation}
Similarly, the fundamental group of the base space acts naturally coordinate-wise 
on the homotopy fibre. The monodromy representation of this fibration is the 
main object of study in this article.

Note that the homotopy type of a polyhedral product depends only on the relative 
homotopy type of the pairs $(X,A)$, as observed in \cite{denham}. 
We are mainly interested in the cases when $G_1,\dots,G_n$ are finite discrete groups. 
If  the pairs $(\underline{EG},\underline{G})$ 
are replaced by $(\underline{I},\underline{F})$, where $I$ is the unit interval and $F_i\subset I$ 
has the cardinality of $G_i$, then there is a homotopy equivalence 
$(\underline{EG},\underline{G})^K\simeq (\underline{I},\underline{F})^K$. 
Moreover, if $K=K^0$ is the 0-skeleton of $K$, 
then it follows from Example \ref{example: polyhedral products} that 
$(\underline{BG},\underline{1})^{K^0}=BG_1 \vee \cdots \vee BG_n$.  
The homotopy fibre  $(\underline{I},\underline{F})^{K^0}$ has the 
homotopy type of a finite wedge of circles
$(\underline{I},\underline{F}) ^{K^0}\simeq \bigvee_{\rho_{K^0}} S^1, $ 
as shown in \cite{stafa.fund.gp}, where
\begin{equation}\label{eqn: rank of kernel for chordal graph}
\rho_{K^0}=(n-1)\prod_{i=1}^n |G_i| - \sum_{i=1}^n (\prod_{j\neq i} |G_j|)+1.
\end{equation}
This gives a topological proof of a classical theorem of J. Nielsen 
\cite{nielsen} concerning the rank of the free group in the following 
short exact sequence of groups
$$ 1 \to F_{\rho_{K^0}} \to G_1\ast \cdots \ast G_n \to \prod_{1\leq i \leq n} G_i \to 1.$$
Therefore, the rank $\rho_{K^0}$ depends only on the order of the
groups $G_1,\dots,G_n$, and not on their group structure. 
For simplicity we simply write $\rho_{K^0}$ for
the rank, when the orders of $G_i$ are clear from the context.

More generally, it was shown in \cite{stafa.fund.gp} that if $G_1,\dots,G_n$ are countable discrete 
groups then the spaces in the fibration sequence (\ref{eqn: D-S-fibration})
are Eilenberg-MacLane spaces of the type $K(\pi,1)$ if and only if $K$ is a flag complex. 
Moreover, for any $K$, the fundamental group of the polyhedral product
$(\underline{BG},\underline{1})^K$ is determined 
by the 1-skeleton $K^1$ and is isomorphic to
$\pi_1((\underline{BG},\underline{1})^K) \cong \prod_{K^1} G_i$, 
the graph product of the groups $G_1,\dots,G_n$.
Hence we obtain a short exact sequence of groups, which is true also for $K$
not necessarily a flag complex:
\begin{equation}\label{eqn: s.e.s.}
 1 \to \pi_1((\underline{EG},\underline{G})^K) \to 
	\prod_{K^1} G_i \to \prod_{1\leq i \leq n} G_i \to 1.
\end{equation}

We want to study simplicial complexes $K$ for which the kernel of the 
short exact sequence above is a free group. The following theorem 
shows exactly which simplicial complexes $K$ have this property.

\begin{thm}\label{thm: fibre is a graph if K^1 is chordal}
Let $G_1,\dots,G_n$ be (countable) discrete groups and $K$ be a flag 
complex on $n$ vertices. Then $(\underline{EG},\underline{G})^K$ 
has the homotopy type of a graph if and only if $K^1$ is a chordal graph. 
\end{thm}
\begin{proof}
It was shown in \cite[Theorem 1.1]{stafa.fund.gp} that 
the space $(\underline{EG},\underline{G})^K$ is a $K(\pi,1)$ if and only if
$K$ is a flag complex. Therefore, we get the short exact sequence of groups
in (\ref{eqn: s.e.s.}). Panov and Veryovkin \cite[Theorem 4.3]{panov2016polyhedral}
showed that $\pi_1((\underline{EG},\underline{G})^K)$ is free if and only if 
the graph $K^1$ is a chordal graph, which completes the proof. See also
\cite[Theorem 4.2]{servatius1989surface} for a relevant result.
\end{proof}

\section{Commutator subgroups of graph groups}\label{sec: commutator subgp}

The \textit{commutator subgroup} $[G,G]$ of a group $G$ is generated by  
commutators $[g,h]:=ghg^{-1}h^{-1}$ with $g,h\in G$.
Let $G_1,\dots,G_n$ be finite groups and $K$ be a flag complex with $K^1$  
a chordal graph. From the previous section we know that under these assumptions 
the kernel of the projection map 
\begin{equation}\label{eqn: projection p}
p: \prod_{K^1} G_i \to \prod_{1\leq i \leq n} G_i
\end{equation}
is a free group. Denote the rank of $\ker(p)$ by $\rho_K$. This kernel is generated by iterated commutators of the form
$$[g_{j_1},[g_{j_2},[\dots,[g_{j_k},g_{j_{k+1}}]\dots]]],$$
where $g_{j_i}$ belong to distinct $G_{j_i}$. 
The kernel ${\rm ker}(p)$ is not necessarily the commutator subgroup of 
$\prod_{K^1} G_i$, if at least one of the $G_i$ is not abelian.
However, ${\rm ker}(p)$ {coincides} with the commutator subgroup of the 
graph group if the groups $G_1,\dots,G_n$ are all abelian.

In this section we describe a basis for the free group 
${\rm ker}(p)=F_{\rho_K}$ in terms of iterated commutators.
A basis was given in \cite[Lemma 4.7]{panov2016polyhedral}, 
where the groups under consideration had order 2, that is the graph 
groups were right-angled Coxeter groups. Another version 
of this basis of commutators was studied by Grbi\'c, Panov, Theriault, and Wu
in the context of exterior algebras in \cite[Theorem 4.3]{wu.grbic.panov}.

Before we proceed it is important to note that the commutator subgroup
of a free group can also be described by a generating set not consisting
of commutators. 
One can obtain new presentations not involving commutators 
using Tietze transformations 
\cite{magnus2004combinatorial,vdovina1995constructing}.

Recall that the fibre in (\ref{eqn: D-S-fibration})
depends only on the order of the finite groups $G_1,\dots,G_n$, since its homotopy type
depends on the relative homotopy type of the pairs $(EG_i,G_i)$ 
(this is true for any polyhedral product -- see \cite[p.31]{denham}). Therefore, 
it suffices to describe the basis elements (i.e. iterated commutators)
of ${\rm ker}(p)=F_{\rho_K}$ only when $G_1,\dots,G_n$ are cyclic groups.
The basis for the general case of any finite groups $G_1,\dots,G_n$
can be obtained by considering the basis when $G_i$ are all cyclic
and then replacing the entries in the commutators with the nontrivial elements of $G_i$.
This observation will be used in Section \ref{sec: graph products abelian gps}.

\begin{prop}\label{prop: basis of pi_1 for K^0}
Let $G_1,\dots,G_n$ be finite groups and $K=K^0$.
Then the fundamental group of the fibre $(\underline{EG},\underline{G})^K$ 
in (\ref{eqn: D-S-fibration}) is the free group with basis consisting of the 
following iterated commutators
\begin{equation}
[g_j,g_i],\,\, [g_{k_1},[g_j,g_i]],\,\,\dots\,,[g_{k_1},\dots,[g_{k_l},[g_j,g_i]]\dots],
\end{equation}
where $g_t\in G_t$, with $k_1<\cdots<k_{l}<j$ and $j>i\neq k_r$, for all $r$. 
\end{prop}

\begin{proof}
We need to show that (1) this set of elements generates the fundamental group, and that
(2) the number of elements in the set equals the rank of the free group
in equation (\ref{eqn: rank of kernel for chordal graph}).
Since the first part of the proof is essentially the proof of 
\cite[Lemma 4.7]{panov2016polyhedral}, we only give an outline here. 
First recall the \textit{Hall identities} for group elements $a,b,c$
\begin{equation}\label{eqn: Hall-Witt}
[a,bc]=[a,c][a,b][[a,b],c], \text{ and }
[ab,c]=[a,c][[a,c],b][b,c],
\end{equation}
and if $x$ is a commutator we can write
\begin{equation}\label{eqn: Hall-Witt derived}
[g_j,[g_i,x]]=[g_j,x] [x,[g_i,g_j]]  [g_j,g_i]  [x,g_i]  [g_i,[g_j,x]]  [x,g_j]  [g_i,g_j] [g_i,x].
\end{equation}
Therefore, given the equations (\ref{eqn: Hall-Witt}) and (\ref{eqn: Hall-Witt derived})
we proceed as follows:
\begin{itemize}[leftmargin=.2in]
\item[--] we can use the identities above to switch between the commutators 
		$[g_j,[g_i,x]]$ and $[g_i,[g_j,x]]$ by using other commutators of lower degrees,
		we can change the order of $g_{k_1},\dots,g_{k_l},g_j$ in the commutator
		$[g_{k_1},\dots,[g_{k_l},[g_j,g_i]]\dots]$ to have them in increasing order, so 
		we can thus obtain the inequalities in the proposition;
\item[--] we can use the identities to eliminate commutators with two entries from the same 
		group, since we can reorder the terms to have these two entries next to each other,
		then their product is in the same group, hence having a lower degree commutator. 
		For example, from (\ref{eqn: Hall-Witt}) we can write $[a,[b,c]]$ in terms of 
		$[ca,b]$ and other lower degree comutators, and if $a,c$ are from the 
		same group, then we have reduced the degree of the original commutator  
		(also verifying the statement following (\ref{eqn: projection p}) about the kernel for $K^0$);
\item[--] we can thus assume that the commutators have the prescribed order, and that $g_k,g_l$
		are from different groups if $k\neq l$;
\item[--] finally we obtain a generating set for the the free group $F_{\rho(n)}$ in terms of commutators
		$[g_{k_1},\dots,[g_{k_l},[g_j,g_i]]\dots]$, with $k_1<\cdots<k_{l}<j$ 
		and $j>i\neq k_r$, for all $r$. Call this set $\mathcal S$.
\end{itemize}
Note that $\mathcal S$ does not generate the commutator subgroup, unless all $G_i$ are abelian.
Now we need to show that this generating set is minimal, that is $|\mathcal S|=\rho_K$
defined in (\ref{eqn: rank of kernel for chordal graph}).
For this we use induction on the number $n$ of vertices of $K$.
Let us denote $\rho_K=\rho(n)$ when $K=K^0$ consists of only $n$ vertices.

Assume $G_1,\dots,G_n$ have orders $m_1,\dots,m_n$, respectively.
For $n=2$, clearly $|\mathcal S|=|\{(g_j,g_i)|g_j,g_i\neq 1\}|=(m_1-1)(m_2-1)=\rho_K=\rho(2)$.
Suppose this is true for $n=k$. For $n=k+1$ we claim that 
$$
\rho(k+1)=m_{k+1} \rho(k) +(m_{k+1}-1)(\prod_{1\leq i \leq k}m_i - 1).
$$
When we introduce a new group $G_{k+1}$, since it has the highest index,
according to our assumption, its elements come only second from the last
in the iterated commutator.
This yields $m_{k+1} \rho(k)$ generators, by taking a commutator and placing the elements of $G_{k+1}$
second from last in the iterated commutators, giving a higher degree commutator.
The insertion of a new non-trivial element, gives more freedom to the last element in the iterated
commutator. For each non-trivial element $1\neq g_i \in G_{k+1}$, the last entry
can take $\prod_{1\leq i \leq k}m_i - 1$ values. 
Now counting for each element of the new group, 
gives the second term in our claim, hence proving the claim.
Combining equation (\ref{eqn: rank of kernel for chordal graph}) and the claim,
and rearranging the terms, the minimality of $\mathcal S$ follows.
\end{proof}

\begin{ex}
Let $G_1=\Z_2=\{1,x\},\, G_2=\Z_2=\{1,y\},\,G_3=\Z_3=\{1,z,z^2\},$ 
and $K=\{\{1\},\{2\},\{3\}\}.$ Then the fibre of the fibration
has fundamental group the free group $F_9$ with a minimal generating set 
$\mathcal S$ given by
$$
\mathcal S = \{[z,x] , [z^2,x] , [z,y] , [z^2,y] , [y,x], [x,[z,y]] , [x,[z^2,y]] , [y,[z,x]] , [y,[z^2,x]]\}.
$$
\end{ex}

For simplicial complexes $K$ strictly larger than their 0-skeleton the 
following proposition holds.

\begin{prop}\label{prop: basis of pi_1 for K=flag}
Let $G_1,\dots,G_n$ be finite groups and $K$ be a flag complex 
with $n$ vertices such that $K^1$ is a chordal graph.
Then the fundamental group of the fibre in (\ref{eqn: D-S-fibration}) is a free group
with a basis the iterated commutators
$$
[g_j,g_i],\,\, [g_{k_1},[g_j,g_i]],\,\,\dots\,,[g_{k_1},\dots,[g_{k_l},[g_j,g_i]]\dots],
$$
where $g_t\in G_t$, with $k_1<\cdots<k_{l}<j$ and $j>i\neq k_r$, for all $r$, and
$i$ is the smallest vertex in a component not containing $j$ in the subcomplex
of $K$ restricted to $\{k_1,\dots,k_{l},j,i\}.$ 
\end{prop}

When we start introducing edges in $K^0$, then we start introducing commutator relations
$[g_i,g_j]$ whenever $\{i,j\}$ is an edge. In the iterated commutator $[g_{k_1},\dots,[g_{k_l},[g_j,g_i]]\dots]$
if $i,j$ are in the same connected component of $K$ restricted to $\{k_1,\dots,k_{l},j,i\},$ 
then there is a path from $i$ to $j$ with coordinates from $\{k_1,\dots,k_{l}\},$ hence
we can consider the iterated commutator induced by these vertices. Using
relations from the edges we can reduce this commutator to another commutator
of shorter length etc. Thus we can choose $i,j$ to be in different path components.
If we have two commutators where the last coordinate is in the same
component, one can show that we can write one in terms of the other.
Hence, we choose the smallest between them.
We leave it to the reader to check that the detailed arguments
in the proof of \cite[Theorem 4.5]{panov2016polyhedral} 
work also for any selection of finite groups.

\begin{ex}
Let us consider an example with the symmetric group on 3 letters.
Let $G_1=\Sigma_3:=\langle s,t | s^2=t^2, (st)^3\rangle=\{1,x_1,x_2,x_3,x_4,x_5\}$,
$G_2=\Z_2=\{1,y\}$, $G_3=\Z_3=\{1,z,z^2\},$ and $K''=\{\{1,2\},\{3\}\}$ in 
Figure \ref{fig: K and K'}. Then the fibre of the fibration
has fundamental group the free group $F_{22}$ with $\mathcal S$ given by
\begin{align*}
\mathcal S = \{	& [z,x_1] , [z,x_2], [z,x_3], [z,x_4],[z,x_5], 
			[z^2,x_1] , [z^2,x_2],[z^2,x_3],[z^2,x_4],[z^2,x_5], \\
			& [z,y] , [z^2,y] , [x_1,[z,y]] ,[x_2,[z,y]],  [x_3,[z,y]],  [x_4,[z,y]],  [x_5,[z,y]] \\
			& [x_1,[z^2,y]] ,[x_2,[z^2,y]],  [x_3,[z^2,y]],  [x_4,[z^2,y]],  [x_5,[z^2,y]]\}.
\end{align*}
Note that the structure of the symmetric group $\Sigma_3$ 
was not needed to write the generating set $\mathcal S$. 
Therefore, if we replace $\Sigma_3$ with the cyclic group of
order six $C_6$, then the corresponding generating set $\mathcal S$ has the 
same number and types of generators where in the commutators in $S$ we 
replace the elements of the symmetric group with those of the cyclic group.
\end{ex}

\section{Examples of monodromy representations}\label{section: e.g.}

Let $G_1,\dots,G_n$ be finite discrete groups and $K$ be a flag complex with $K^1$ 
a chordal graph. Consider the following commutative diagram 
\begin{equation}\label{fig: comm diag. for Aut and Out}
\begin{tikzcd}
 1 \arrow{r} & F_{\rho_K} \arrow{r} \arrow{d}{\Psi = \rm iso} & 
 								\prod_{K^1} G_i \arrow{r} \arrow{d}{\Theta_K} 
 	& \prod_{i} G_i\arrow{r} \arrow{d}{\Phi_{K}} & 1 \\
 1 \arrow{r} & \Inn(F_{\rho_K}) \arrow{r} &\Aut(F_{\rho_K}) \arrow{r} & 
 								\Out(F_{\rho_K}) \arrow{r} & 1,
\end{tikzcd}
\end{equation}
where $\Theta(g)(h)=ghg^{-1}$ and $\Psi(g)(h)=ghg^{-1}$. We are interested
in describing the maps $\Theta_{K}$ and $\Phi_{K}.$

For examples concerning only two finite groups, i.e. $n=2$, see \cite{stafa.monodromy}, 
where explicit answers are given. 
We can explicitly describe faithful representations (eg. by using \verb|Magma|)
$$\Phi_K: G_1\times \cdots \times G_n \to \SL({\rho_K},\Z),$$
where $G_1, \dots , G_n$ are finite abelian groups and $n\geq 3$. 
In general, if $G_i$ are any finite groups (not necessarily abelian), 
we obtain faithful representations of graph products of finite groups
$$ \Theta_K: \prod_{K^1} G_i \to \Aut(F_{\rho_K})$$
as well as faithful monodromy representations of direct products of finite groups
$$ \Phi_K: G_1\times \cdots \times G_n \to \GL({\rho_K},\Z)$$
as will be shown below.
These include many interesting classes of discrete groups, such as right-angled 
Coxeter groups. If one of the groups is infinite discrete, then additional
examples include hyperbolic groups as described in \cite[Theorem 5.1]{holt2012generalising}, 
braid groups, right-angled Artin groups and more. 
Thus such representations can be realized as monodromy representations.

The rank $\rho_K$ increases {very fast} (\ref{eqn: rank of kernel for chordal graph}) 
if we increase the order and the number of the groups in consideration. 
We concentrate on a couple of examples including right-angled Coxeter groups. 
In addition we select the
simplicial complexes $K$ and $K'$ in Figure \ref{fig: K and K'}, 
to keep the rank of the free group small.
It is certainly possible to obtain many more explicit examples, which we leave to the interested reader.
However, note that the basis generated in \verb|Magma| 
is different (yet equivalent) from the basis we describe in Section 
\ref{sec: commutator subgp}. To do the following examples it suffices to have
Theorem \ref{thm: fibre is a graph if K^1 is chordal} and the basis generated by
\verb|Magma|, but we need Propositions \ref{prop: basis of pi_1 for K^0}
and \ref{prop: basis of pi_1 for K=flag} to have an explicit basis in general.
\begin{figure}[ht!]
\begin{tikzpicture}
[scale=1, vertices/.style={draw, fill=black, circle, inner sep=0.5pt}]
\node[vertices, label=below:{{\footnotesize$$1$$}}] (f) at (0,0) {};
\node[vertices, label=below:{{\footnotesize$$2$$}}] (e) at (0.5,0) {};
\node[vertices, label=below:{{\footnotesize$$3$$}}] (g) at (1,0) {};

\node[vertices, label=below:{{\footnotesize$$1$$}}] (a) at (3,0) {};
\node[vertices, label=below:{{\footnotesize$$2$$}}] (b) at (3.5,0) {};
\node[vertices, label=below:{{\footnotesize$$3$$}}] (c) at (4,0) {};
\node[vertices, label=below:{{\footnotesize$$4$$}}] (d) at (4.5,0) {};

\node[vertices, label=below:{{\footnotesize$$1$$}}] (k) at (6,0) {};
\node[vertices, label=below:{{\footnotesize$$2$$}}] (l) at (6.5,0) {};
\node[vertices, label=below:{{\footnotesize$$3$$}}] (m) at (7,0) {};
\foreach \to/\from in {a/b,b/c,c/d}
\draw [-] (\to)--(\from);

\foreach \to/\from in {k/l}
\draw [-] (\to)--(\from);
\end{tikzpicture}
\caption{$K,\,\,K',$ and $K''$}
\label{fig: K and K'}
\end{figure}
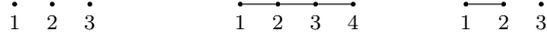

\begin{ex}\label{ex: 1}
Consider three groups of order 2 and the following short exact sequence obtained from
the fibration sequence (\ref{eqn: D-S-fibration})
$$
1 \to F_5 \to \Z_2\ast \Z_2, \ast \Z_2 \to \Z_2 \times \Z_2, \times \Z_2 \to 1,
$$
corresponding to the simplicial complex $K$ in Figure \ref{fig: K and K'},
where each of the cyclic groups is generated by $a$, $b$ and $c$, respectively. 
Recall that the rank of the fibre is given by equation 
(\ref{eqn: rank of kernel for chordal graph}) and in this case is 5.
Then $ F_5$ has a generating set 
(thus a presentation) given by
$$
F_5:=\langle (ba)^2,(ca)^2,(cb)^2,acbcba,bcacba \rangle=\langle x_1,x_2,x_3,x_4,x_5\rangle.
$$
The action of $\Z_2 \times \Z_2, \times \Z_2$ on $F_5$ is determined by the following:
\begin{align*}
a\cdot x_i = \begin{cases}
x_1^{-1}		\\
x_2^{-1} 		\\
x_4			\\
x_3			\\
x_5^{-1}		
\end{cases}
b\cdot x_i = \begin{cases}
x_1^{-1}			\\
x_5 x_1^{-1} 		\\
x_3^{-1}			\\
x_1 x_4^{-1} x_1^{-1}	\\
x_2 x_1^{-1}		
\end{cases}
c\cdot x_i = \begin{cases}
x_3 x_5 x_4^{-1} x_2^{-1}		& \text{if } i=1\\
x_2^{-1} 					& \text{if } i=2\\
x_3^{-1}					& \text{if } i=3\\
x_2 x_4^{-1} x_2^{-1}			& \text{if } i=4\\
x_3 x_1 x_4^{-1} x_2^{-1}		& \text{if } i=5,
\end{cases}
\end{align*}
where $a,b,c$ act by conjugation. This gives also a faithful representation of 
the right-angled Coxeter group to the automorphism group $\Aut(F_5)$.
It is straightforward to check that the induced action on 
the abelianization $\Z^5$ gives a faithful representation of 
the right-angled Coxeter group to the special linear group over the integers:
\begin{align*}
\Phi_K: \Z_2 \times \Z_2, \times \Z_2 & \to \SL(5,\Z)\\
			a,b,c &\mapsto X,Y,Z,
\end{align*}
where $X,Y,Z$ are the following matrices, respectively:
{{{{ \small
\begin{align*}
\left(\begin{array}{rrrrr}
-1 & 0  & 0 & 0 & 0 \\
0  & -1 & 0 & 0 & 0 \\
0  & 0  & 0 & 1 & 0 \\
0  & 0  & 1 & 0 & 0 \\
0  & 0  & 0 & 0 &-1 \\
\end{array}\right), 
\left(\begin{array}{rrrrr}
-1  & 0  & 0  & 0  & 0 \\
-1  & 0  & 0  & 0  & 1 \\
0   & 0  & -1 & 0  & 0 \\
0   & 0  & 0  & -1 & 0 \\
-1  & 1  & 0  & 0  & 0 \\
\end{array}\right), 
\left(\begin{array}{rrrrr}
0  & -1  &  1  &  -1 & 1 \\
0  & -1  &  0  &  0  & 0 \\
0  &  0  & -1  &  0  & 0 \\
0  &  0  &  0  & -1  & 0 \\
1  & -1  &  1  & -1  & 0 \\
\end{array}\right).
\end{align*}
}}}}

Note that the generators of the fundamental group of 
the polyhedral product $({E\Z_2},{\Z_2})^K$ 
can be described using the loops in Figure \ref{fig: the loops}
sitting in the space $({I},{F})^K$.

\begin{figure}[ht!]
\begin{tikzpicture}[scale=.5]
\draw[line width=0.3mm](0,0)--(1,1);
\draw[-,line width=0.3mm](1,1)--(3,1);
\draw[line width=0.3mm](3,1)--(2,0);
\draw[line width=0.3mm](2,0)--(0,0);
\draw[dotted](0,0)--(0,2);
\draw[dotted](0,2)--(1,3);
\draw[dotted](1,3)--(3,3);
\draw[dotted](3,3)--(2,2);
\draw[dotted](2,2)--(0,2);
\draw[-,dotted](1,1)--(1,3);
\draw[dotted](2,2)--(2,0);
\draw[dotted](3,3)--(3,1);
\end{tikzpicture}
\hspace{.2in}
\begin{tikzpicture}[scale=.5]
\draw[dotted](0,0)--(1,1);
\draw[-,dotted](1,1)--(3,1);
\draw[dotted](3,1)--(2,0);
\draw[dotted](1,1)--(1,3);
\draw[line width=0.3mm](2,0)--(0,0);
\draw[line width=0.3mm](0,0)--(0,2);
\draw[dotted](0,2)--(1,3);
\draw[dotted](1,3)--(3,3);
\draw[dotted](3,3)--(2,2);
\draw[line width=0.3mm](2,2)--(0,2);
\draw[line width=0.3mm](2,2)--(2,0);
\draw[dotted](3,3)--(3,1);
\end{tikzpicture}
\hspace{.2in}
\begin{tikzpicture}[scale=.5]
\draw[line width=0.3mm](0,0)--(1,1);
\draw[line width=0.3mm](1,1)--(1,3);
\draw[-,dotted](1,1)--(3,1);
\draw[dotted](3,1)--(2,0);
\draw[dotted](2,0)--(0,0);
\draw[line width=0.3mm](0,0)--(0,2);
\draw[line width=0.3mm](0,2)--(1,3);
\draw[dotted](1,3)--(3,3);
\draw[dotted](3,3)--(2,2);
\draw[dotted](2,2)--(0,2);
\draw[dotted](2,2)--(2,0);
\draw[dotted](3,3)--(3,1);
\end{tikzpicture}
\hspace{.2in}
\begin{tikzpicture}[scale=.5]
\draw[line width=0.3mm](0,0)--(1,1);
\draw[-,line width=0.3mm](1,1)--(3,1);
\draw[dotted](3,1)--(2,0);
\draw[line width=0.3mm](2,0)--(0,0);
\draw[dotted](0,0)--(0,2);
\draw[dotted](0,2)--(1,3);
\draw[dotted](1,3)--(3,3);
\draw[line width=0.3mm](3,3)--(2,2);
\draw[dotted](2,2)--(0,2);
\draw[-,dotted](1,1)--(1,3);
\draw[line width=0.3mm](2,2)--(2,0);
\draw[line width=0.3mm](3,3)--(3,1);
\end{tikzpicture}
\hspace{.2in}
\begin{tikzpicture}[scale=.5]
\draw[line width=0.3mm](0,0)--(1,1);
\draw[-,line width=0.3mm](1,1)--(1,3);
\draw[-,dotted](1,1)--(3,1);
\draw[line width=0.3mm](3,1)--(2,0);
\draw[line width=0.3mm](2,0)--(0,0);
\draw[dotted](0,0)--(0,2);
\draw[dotted](0,2)--(1,3);
\draw[line width=0.3mm](1,3)--(3,3);
\draw[dotted](3,3)--(2,2);
\draw[dotted](2,2)--(0,2);
\draw[dotted](2,2)--(2,0);
\draw[line width=0.3mm](3,3)--(3,1);
\end{tikzpicture}
\caption{The generators $x_1,\,x_2,\,x_3,\,x_4,$ and $x_5$, respectively}\label{fig:t3}
\label{fig: the loops}
\end{figure}
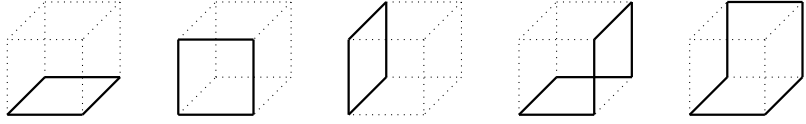
\end{ex}

\begin{ex}\label{ex: 2}
Now we consider four cyclic groups. 
Construct the right-angled Coxeter group over 
the simplicial complex $K'$ given in Figure \ref{fig: K and K'}. 
Then equation (\ref{eqn: D-S-fibration}) gives the following short 
exact sequence of groups
$$
1 \to F_{5} \to \prod_{K'}\Z_2   \to   \Z_2^4 \to 1,
$$
where $F_{5} = \langle x_1,\dots,x_{5}\rangle:=\langle     
    (ca)^2,
    (da)^2,
    (db)^2,
    a d b d b a,
    c d a d c a,
    \rangle.$ 
The conjugation action is then described as follows:
\begin{align*}
a\cdot x_i = \begin{cases}
x_1^{-1}		\\
x_2^{-1} 		\\
x_4			\\
x_3			\\
x_5^{-1}		
\end{cases}
b\cdot x_i = \begin{cases}
x_1			\\
x_3^{-1}x_2 x_4 	\\
x_3^{-1}		\\
x_4^{-1}		\\
x_3^{-1} x_5 x_4		
\end{cases} \, 
c\cdot x_i = \begin{cases}
x_1^{-1}		\\
x_5 x_1^{-1} 	\\
x_3			\\
x_1 x_4 x_1^{-1}	\\
x_2 x_1^{-1}		
\end{cases}\!\!
d\cdot x_i = \begin{cases}
x_5 x_2^{-1}		\\
x_2^{-1} 			\\
x_3^{-1}			\\
x_2 x_4^{-1} x_2^{-1}	\\
x_1 x_2^{-1}		
\end{cases}
\end{align*}
for all values of $i=1,2,3,4,5,$ respectively. 
Then there is a representation of the right-angled Coxeter group 
$\prod_{K'}\Z_2$ into the automorphism group $\Aut(F_5)$.
$$ \prod_{K'}\Z_2 \hookrightarrow \Aut(F_5) .$$
The induced action on the abelianization of the free group gives a
faithful representation
\begin{align*}
\Phi_{K'}: \Z_2^4 & \to \SL(5,\Z)\\
			a,b,c,d &\mapsto A,B,C,D
\end{align*}
where $A,B,C,D$ are the following matrices, respectively:
{{{{{ \small
\begin{align*}
A=&\left(\begin{array}{rrrrr}
1 & 0 & 0 & 0 & 0 \\ 
0 & 1 & -1 & 1 & 0 \\ 
0 & 0 & -1 & 0 & 0 \\ 
0 & 0 & 0 &-1 & 0 \\ 
0 & 0 & -1 & 1 & 1 \\ 
\end{array}\right), 
B=\left(\begin{array}{rrrrr}
-1 & 0 & 0 & 0 & 0 \\ 
0 & -1 & 0 & 0 & 0 \\ 
0 & 0 & 0 & 1 & 0 \\ 
0 & 0 & 1 &0 & 0 \\ 
0 & 0 & 0 & 0 & -1 \\ 
\end{array}\right), \\
C=&\left(\begin{array}{rrrrr}
-1 & 0 & 0 & 0 & 0 \\ 
-1 & 0 & 0 & 0 & 1 \\ 
0 & 0 & 1 & 0 & 0 \\ 
0 & 0 & 0 &1 & 0 \\ 
-1 & 1 & 0 & 0 & 0 \\ 
\end{array}\right), 
D=\left(\begin{array}{rrrrr}
0 & -1 & 0 & 0 & 1 \\ 
0 & -1 & 0 & 0 & 0 \\ 
0 & 0 & -1 & 0 & 0 \\ 
0 & 0 & 0 &-1 & 0 \\ 
1 & -1 & 0 & 0 & 0 \\ 
\end{array}\right).
\end{align*}
}}}}}
\end{ex}

One can start with any finite groups $G_1,\dots , G_n$ with given 
presentations and any simplicial complex $K$ with on $n$ vertices 
such that $K^1$ is a chordal graph. If either of the
groups is not abelian, then the representations obtained
in the abelianization may not have images in $\SL(\rho_K,\Z)$, 
but rather in $\GL(\rho_K,\Z)$ as shown in \cite[Example 2]{stafa.fund.gp}.

\section{Graph products of abelian groups}\label{sec: graph products abelian gps}

Every finite abelian group can be written as a 
finite direct sum of finite cyclic subgroups
with order a power of a prime. Here we will 
describe how to think of a graph product of 
finite abelian groups over $K$ as a graph product of 
cyclic groups over a new simplicial complex $K_{\underline{G}}$, 
at the expense of having more vertices in the simplicial complex.

In \cite[Theorem 2.2]{stafa.monodromy} it was 
shown that two cyclic subgroups yield a faithful
monodromy representation 
$$C_n\times C_m \to \Out(F_{\rho_K})$$ 
and a faithful representation 
$$\Phi_K: C_n\times C_m \to \SL({\rho_K},\Z),$$ 
where $K=\{\{1\},\{2\}\}$. 

\

Consider two finite abelian groups $G,H$.
Then we can write them as direct sums if cyclic groups
$$G\cong \bigoplus_{i\in I} C_{n_i},\,\,\,\, H \cong \bigoplus_{j\in J} C_{m_j}.$$
We can then replace the simplicial complex $K=\{\{1\},\{2\}\}$ 
by the union of two simplices 
$$K'=\Delta[|I|-1] \sqcup \Delta[|J|-1].$$ 
This does not change the monodromy
representation $G\times H \to \Out(F_{\rho_K})$, because the short exact sequences
of groups
$$
1 \to F_{\rho_K} \to G \ast H   \to   G \times H \to 1,
$$
and
$$
1 \to F_{\rho_{K'}} \to \prod_{(K')^1} C_{n_i,m_j}   \to   \prod_{i,j} C_{n_i,m_j} \to 1,
$$
are equivalent; note that $\prod_{i,j} C_{n_i,m_j} \cong G\times H$,
and 
$$\prod_{(K')^1} C_{n_i,m_j}\cong \prod_{(\Delta[|I|-1])^1}C_{n_i} 
\ast \prod_{(\Delta[|J|-1])^1}C_{m_j}\cong G\ast H.$$

More generally, if $K$ is a flag complex with $K^1$ a chordal graph, 
and $G_1,\dots,G_n$ are finite abelian,  write the 
direct sum decomposition of each group $G_i$ 
$$G_i\cong \bigoplus_{j\in J_i} C_{n_j^i}.$$
Replace each vertex $i$ on the chordal graph $K^1$
by the simplex $\Delta[|J_i|-1]$. Give unique names to all vertices
in all these different simplices. Note that, if $\{i,j\}$ is an edge in $K$,
and $\Delta[|J_i|-1]=\{\{v_0^i,\cdots,v_{|J_i|-1}^i\}\}$, then we need 
to add $\{v_k^i,v_t^j\}$ to the new simplicial complex for all $k,t$
since the $G_i,G_j$ commute if and only if all their subgroups commute. 
We can now define the following simplicial complex. 

\begin{defn}
Let $\underline{G}:=\{G_1,\dots,G_n\}$ be finite abelian groups. 
Let $K$ be a simplicial complex on $n$ vertices with 1-skeleton
$K^1$ a chordal graph. Define the \textbf{simplicial complex $K_{\underline{G}}$} 
to be the flag complex obtained from $K$ by the following procedure: 
replace each vertex $i$ of $K$ with the full simplex 
$\Delta[|J_i|-1]=\{\{v_0^i,\cdots,v_{|J_i|-1}^i\}\}$,
add an edge between the vertices in $\Delta[|J_k|-1]$ and the vertices
in $\Delta[|J_l|-1]$ if $G_k$ and $G_l$ commute in the graph product
$\prod_{K^1} G_i$, and take the
corresponding clique complex of the 1-skeleton 
of this new simplicial complex. 
\end{defn}

We then have the following lemma.

\begin{lemma}\label{lem: new flag complex from given K}
With the same assumptions, the graph $(K_{\underline{G}})^1$ is chordal. 
\end{lemma}
\begin{proof}
This follows from the definition: 
The 1-skeleton of $\Delta[|J_i|-1]$ and $K^1$ are chordal graphs.
If $c$ is a cycle of length greater than 3, then its edges are either
all in $(\Delta[|J_i|-1])^1$ for some $i$ or it moves between 
various 1-skeleta $(\Delta[|J_k|-1])^1$. Suppose $c$ has length 4.
If vertices of $c$ are all in a single $(\Delta[|J_i|-1])^1$ we are done.
If vertices of $c$ lie in two distinct $(\Delta[|J_i|-1])^1$'s, 
then there is one edge between $(\Delta[|J_k|-1])^1$ and
$(\Delta[|J_l|-1])^1$, thus there is an edge between all the vertices
between these two simplices, in particular between nonconsecutive vertices. 
If vertices of $c$ lie in three distinct $(\Delta[|J_i|-1])^1$'s, 
the same argument holds. 
If vertices of $c$ lie in four distinct $(\Delta[|J_i|-1])^1$'s, then 
$c$ is a replica of a cycle in $K$.
The same arguments show the triangulation of longer cycles $c$ .
\end{proof}

\begin{thm}
Let $G_1,\dots, G_n$ be finite abelian groups and $K^1$ a chordal graph. 
Then the faithful monodromy representation 
$G_1\times\cdots \times G_n \to \Out(F_{\rho_K})$ 
induces a faithful representation
$$\Phi_K: G_1\times\cdots \times G_n \to \SL({\rho_K},\Z).$$
\end{thm}
\begin{proof}
By Lemma \ref{lem: new flag complex from given K}, 
the graph $(K_{\underline{G}})^1$  is chordal and 
by definition $K_{\underline{G}}$ is a flag complex. Therefore 
the spaces in fibration (\ref{eqn: D-S-fibration}) 
are Eilenberg-MacLane spaces. Furthermore, the monodromy representation 
$$\Phi_K:  G_1\times\cdots \times G_n \to \Out(F_{\rho_K})$$
is equivalent to the monodromy representation 
$$\Phi_{K_{\underline{G}}}:  G_1\times\cdots \times G_n \to \Out(F_{\rho_{K_{\underline{G}}}}),$$
where $F_{\rho_{K_{\underline{G}}}}=F_{\rho_{K}}$
and we rewrite  $$G_i\cong \bigoplus_{j\in J_i} C_{n_j^i}.$$
Each element in $G_i$ lies in
a cyclic group, which by \cite[Theorem 2.2]{stafa.monodromy} maps
faithfully into $\SL(\rho_K,\Z).$
Since $G_1\times\cdots \times G_n$ is abelian the theorem follows.
\end{proof}

\begin{cor}
If $K^1$ is a chordal graph, then there is a faithful 
representation of the graph product $\prod_{K^1} G_i$
into the automorphism group $\Aut(F_{\rho_K})$ of the 
free groups of rank ${\rho_K}$.
In particular, this is true for any right-angled Coxeter group.
\end{cor}
\begin{proof}
This follows by considering the commutative diagram (\ref{fig: comm diag. for Aut and Out}) since the left vertical map is an isomorphism and the right vertical map is an injection.
\end{proof}

Recall that there is a short exact sequence of groups
$$
1 \to \IA_{N} \to \Aut(F_{N}) \to \GL(N,\Z) \to 1
$$
induced by the abelianization of the automorphisms of free groups, 
that is the induced map on the first homology $H_1(F_{N})$. 
The group $\IA_{N}$ is the analogue of the Torelli group in 
mapping class groups of surfaces.
Then we have the following immediate corollary.

\begin{cor}
If $K^1$ is a chordal graph and $G_i$ are finite discrete groups, 
then the images of $\prod_{K^1} G_i$ under the faithful representations above
are not in $\IA_{\rho_K}$.
\end{cor}

\section{Induced maps in homology}\label{sec: induced map in homology}

In this section we prove the following proposition.

\begin{prop}
Let $K$ be a flag complex and $G_1,\dots,G_n$ be finite groups. 
Then the induced map on first homology groups
$$H_1((\underline{EG},\underline{G})^K;\Z) \to H_1((\underline{BG},\underline{1})^K;\Z) $$ 
is the zero map.
\end{prop}
\begin{proof}
The main ingredient in this proof is the fact that the 
abelianizations of both the fundamental group 
$\pi_1((\underline{BG},\underline{1})^K) \cong \prod_{K^1} G_i$ 
and the product $\prod_{1\leq i \leq n} G_i$ are the same. 
For a group $G$ denote $\mathscr A(G):=G/[G,G]$ and the 
abelianization map $G \to \mathscr A(G)$ by ${\rm ab}_G$. 
Note that the abelianization of $\prod_{K^1} G_i$ factors 
through the group $\prod_{1\leq i \leq n} G_i$:
$$
\prod_{K^1} G_i \to \prod_{1\leq i \leq n} G_i 
			\xrightarrow{ab_G} \mathscr A(\prod_{K^1} G_i).
$$
Let
$G:=\prod_{K^1} G_i,\, H:=\prod_{1\leq i \leq n} G_i$, 
and $N=\pi_1 ((\underline{EG},\underline{G})^K).$
Note that in general, for any abelian group $A$ and a surjection 
$h:G \twoheadrightarrow A$, there is a unique map 
$\phi: \mathscr A(G) \twoheadrightarrow A$ such that 
$\phi \circ {\rm ab}_G = h$. 

Consider the following commutative diagram

\begin{center}
\begin{tikzcd}
 1 \arrow{r} & N \arrow{r}{i} \arrow{d}{{\rm ab}_N} & G\arrow{r}{p} \arrow{d}{{\rm ab}_G} & H \arrow{r} \arrow{d}{{{\rm ab}_H}} & 1, \\
 {}  & {\mathscr{A}(N)}\arrow{r}{f}  & \mathscr{A}(G) \arrow{r}{=} & \mathscr{A}(H)  & 
\end{tikzcd}
\end{center}
where $p\circ i =1$ (or 0 if $H$ is abelian), and ${\rm ab}_G \circ i = {\rm ab}_H \circ p \circ i = 0$.

Now, since $\mathscr{A}(G)$ is an abelian group, there is a unique map $f: \mathscr A(N) \to \mathscr A(G)$ such that ${\rm ab}_G \circ i = f \circ {\rm ab}_N$. Since ${\rm ab}_G \circ i=0 = f \circ {\rm ab}_N,$ and ${\rm ab}_N$ is clearly not trivial, then $f$ cannot be onto. Actually $f$ is the zero map since the composition in the bottom row 
$\mathscr{A}(N) \to \mathscr{A}(G) \to \mathscr{A}(H)$ is the zero map, and the second map is an isomorphism.
Therefore, $H_1((\underline{EG},\underline{G})^K;\Z) \to H_1((\underline{BG},\underline{G})^K;\Z)$ is the zero map. 
\end{proof}

This proposition is in the spirit of the induced maps in homology 
introduced in the next section, concerning the spaces
$B(2,G)$ and $E(2,G)$ defined below. We seek a similar result
in that case, too.

\section{CT groups and Feit-Thompson theorem}\label{sec: B(2,G) and E(2,G)}

In this section we study commutative transitive groups defined
below, and use some methods from polyhedral products to understand
the interplay between topology and group theory, and 
characterize some group properties using topology.  
For any group $\pi$ the descending central series is given by
a sequence of normal subgroups
$$
\pi=\Gamma^1 \triangleright \Gamma^2  \triangleright
	\cdots \triangleright \Gamma^{n+1} \triangleright \cdots
$$
where inductively $\Gamma^{n+1}=[\pi,\Gamma^{n}]$ for $n\geq 2.$
If $\pi=F_n$ is the free group of rank $n$, then for any topological group
$G$ there is a filtration
$$
{\rm Hom}(F_n/\Gamma^2,G) \subset {\rm Hom}(F_n/\Gamma^3,G) \subset \cdots \subset G^n.
$$
The sequences of spaces given by 
$$B_k(q,G):={\rm Hom}(F_k/\Gamma^q,G)\subset G^k$$
and
$$E_k(q,G):=G\times {\rm Hom}(F_k/\Gamma^q,G) \subset G^{k+1}$$ 
have the structure of simplicial complexes (\cite{fredb2g}), respectively, 
with respective geometric realizations defined as follows
\begin{align*}
B(q,G):= |B_\ast(q,G)|, \text{ and } E(q,G):=|B_\ast(q,G)|. 
\end{align*}
The projections $E_k(q,G) \twoheadrightarrow B_k(q,G)$ induce a fibration
\begin{equation}\label{eqn: fibration of E(q,G) to B(q,G)}
E(q,G) \to B(q,G) \to BG
\end{equation}
and in particular, for $q=2$ we have
\begin{equation}\label{eqn: fibration of E(2,G) to B(2,G)}
E(2,G) \to B(2,G) \to BG.
\end{equation}

The total space $B(2,G)$ and  the homotopy fibre $E(2,G)$ 
were studied by A. Adem, F. Cohen and E. Torres Giese
\cite{fredb2g}. They posed the question 
whether for finite $G$ the space $B(2,G)$ is always a $K(\pi,1)$,  
having showed that these spaces are occasionally $K(\pi,1)$ 
for the case of {\it commutative transitive groups}.
C. Okay \cite{okay,okay2015colimits} gave 
classes of groups for which $B(2,G)$ is not a $K(\pi,1)$, such as
extraspecial 2-groups of order $2^{2n+1}$, for $n\geq2$, hence 
answering their question. A brief survey is given in \cite[\S  9]{stafa.comm.2}.

\begin{defn}\label{defn: CT group}
A group $G$ is \textbf{commutative transitive} or \textit{CT} 
if commutativity is a transitive relation in $G$. That is, if 
$[a,b]=[b,c]=1$, then $[a,c]=1$ for all non-central elements 
$a,\,b,\,c \in G$.
\end{defn}

The class of CT groups played an important role in the 
classification of finite simple groups and were studied by M. Suzuki
\cite{suzuki1957nonexistence,suzuki1986group}, among many others,
who showed that every non-abelian simple CT-group is of even order and 
isomorphic to $\PSL(2, 2^f)$ for some $f\geq2$. 
Finite CT groups have been classified, see for example 
\cite[p. 519, Theorem 9.3.12]{schmidt1994subgroup}.

In particular, if $G$ is a finite CT group with trivial center
then the following is true.
\begin{prop}[{{\cite[Cor. 8.5]{fredb2g}}}]
If there are maximal abelian subgroups $G_1,\dots, G_n$ of $G$ that cover $G$, 
then there is a homotopy equivalence $B(2,G) \simeq \bigvee_i BG_i$. 
\end{prop}  
With the assumptions of this proposition we have the following corollary.
\begin{cor}
$B(2,G)$ has the homotopy type of the polyhedral product
$(\underline{BG},\underline{1})^{K^0}.$
\end{cor}

In what follows $G$ is assumed to be finite.

Using the five term short exact sequence from the 
Lyndon-Hochschild-Serre spectral sequence Adem, Cohen and Torres Giese
\cite[Proposition 7.2]{fredb2g} showed that the non-surjectivity 
of the induced map on first homology of the fibration 
(\ref{eqn: fibration of E(2,G) to B(2,G)})
$$ H_1(E(2,G)) \to H_1(B(2,G))$$
is equivalent to the Feit-Thompson theorem that groups of odd 
order are solvable. Hence the study of the fibration encodes 
fundamental information about the group $G$.
We would like to use polyhedral products, i.e. topology, 
to extract more information about this equivalent form of 
the Feit-Thompson theorem \cite{feit1963chapter}, 
which is algebraic in nature.

\

Let $G$ be a finite CT group with trivial center 
and let $G_1,\dots, G_n$ be its cover by maximal abelian subgroups as above 
($n$ is called the {{\it covering number}} of $G$). 
Since all spaces are $K(\pi,1)$'s, we will move frequently between fundamental groups 
and their classifying spaces. Note that there are two commutative diagrams of 
short exact sequences of groups:

\begin{equation}\label{eqn: first diagram}
\begin{tikzcd}
 \pi_1(E(2,G)) \arrow{r}\arrow{d} & {\mathcal H} \arrow{r} \arrow{d} & {[G,G]} \arrow{d}\\
 \pi_1(E(2,G)) \arrow{r}\arrow{d} & \pi_1(B(2,G)) \arrow{r} \arrow{d}& \pi_1(BG) \arrow{d}\\
 \ast \arrow{r} 	&   H_1(BG) \arrow{r} & H_1(BG),
\end{tikzcd}
\end{equation}
and

\begin{equation}\label{eqn: commutative diagram}
\begin{tikzcd}
  N \arrow{r}\arrow{d} & \pi_1(\underline{EG},\underline{G})^{K^0} \arrow{r} \arrow{d}{i_2} & {[G,G]} \arrow{d}\\
 \pi_1(E(2,G)) \arrow{r}{i_1} \arrow{d}{p_1} & \pi_1(B(2,G)) \arrow{r}{p_2} \arrow{d}{p_3}& \pi_1(BG) \arrow{d}\\
 Q \arrow{r} 	&   \pi_1(\prod_i BG_i) \arrow{r} & H_1(BG),
\end{tikzcd}
\end{equation}
where $Q$ is a finite abelian group and $N$ is a free group (we omit the trivial groups
on each side of the short exact sequences). The existence of the first diagram is clear, 
whereas for the second diagram, even though for CT groups the map $\pi_1(B(2,G)) \to G$ 
does not factor through the product $\prod_i G_i$, the composition $\pi_1(B(2,G)) \to G  \to H_1(G),$ 
being an epimorphism onto an abelian group, factors uniquely through the abelianization 
of $\pi_1(B(2,G))$, which is the direct product $\prod_i G_i$.

The map $p_1$ factors uniquely through the abelianization $H_1(E(2,G))$, hence there is a 
map $q_1$ such that $p_1 = q_1 \circ {\rm ab}$. Hence there is a diagram 
\begin{equation}\label{eqn: small commutative diagram}
\begin{tikzcd}
 \pi_1(E(2,G)) \arrow{r}{i_1} \arrow{d}{{\rm ab}} & \pi_1(B(2,G))  \arrow{d}{p_3}\\
 H_1(E(2,G)) \arrow{d}{q_1} \arrow[dotted]{r}{(i_1)_\ast} &  \prod_i G_i \arrow[equal]{d}\\
 Q \arrow{r}	&   \prod_i G_i ,
\end{tikzcd}
\end{equation}
where the dotted map is the one we are interested in 
(we are not claiming that the lower square commutes). 
By \cite[Proposition 8.8]{fredb2g} the group $\pi_1(E(2,G))$ 
is free, with rank
$$ \mathcal{N}_G=1-|G:Z(G)| + \sum_{1\leq i \leq n} \left( |G:Z(G)|-|G:G_i|\right). $$
Since $Z(G)=\{1_G\}$, by rearranging the terms of $ \mathcal{N}_G$ 
we obtain the following:
\begin{align*}
	\mathcal{N}_G 	&= 1-|G:Z(G)| + \sum_{1\leq i \leq n} \left(  |G:Z(G)|-|G:G_i|\right)\\
				&= 1-|G| + \sum_{1\leq i \leq n} \left( |G|-|G|/|G_i|\right)\\
				&= 1-|G| + n|G| - \sum_{1\leq i \leq n} |G|/|G_i|\\
				&= 1+ (n-1)|G| - \sum_{1\leq i \leq n} |G|/|G_i|.
\end{align*}
Note that this is a more general version of the formula 
for $\rho_K$ in equation (\ref{eqn: rank of kernel for chordal graph}), 
with the special case of $|G|=\prod_i|G_i|$ giving the rank $\rho_K$ 
when $K$ is only a set of $n$ points; let us use the notation
$\rho_K= \rho(n)$ since $K$ becomes irrelevant. In general 
$l.c.m.(|G_1|,\dots,|G_n|) \leq |G|< \prod_i |G_i| $, since the 
groups $G_i$ cover $G$. Actually they divide each other from 
left to right. Since $|G|$ divides $\prod_i |G_i|,$ then 
$\prod_i |G_i|=C |G|$. Therefore, we have 
\begin{align*}
\rho(n)-\mathcal N_G &= (n-1)|G|(C-1) - \sum_{1\leq j\leq n}(|G|(C-1))/|G_j|\\
			& = |G|(C-1)\left( n-1 - \sum_{1\leq j\leq n}1/|G_j| \right).
\end{align*}
Since all $|G_i|\geq 2$, then we have 
$\rho(n)-\mathcal N_G \geq |G|(C-1)(n/2-1).$
If $n=2$ then $G$ is abelian, so assume that $n\geq 3$. 
Then we get $\rho(n)-\mathcal N_G >0$.

\begin{lemma}
Let $G$ be a finite CT group with trivial center. Then $\rho(n)>\mathcal N_G.$
\end{lemma}
\begin{proof}
In addition to the above argument, this is also a direct consequence
of the fact that the index of $N$ in each of the free groups is given by
the following formula \cite[p.16]{lyndon2015combinatorial}:
\begin{align*}
 & [ \pi_1(E(2,G)) : N] = \frac{{\rm rank} (N) - 1}{ {\rm rank} (\pi_1(E(2,G)))-1} = |Q| = \frac{\prod_i|G_i|}{|H_1(G)|},
\text{ and,}\\
 & [\pi_1((\underline{EG},\underline{G})^{K_0}):N] = \frac{{\rm rank} (N) - 1}{ {\rm rank} (\pi_1((\underline{EG},\underline{G})^{K_0})-1} = |[G,G]| = \frac{|G|}{|H_1(G)|}.
\end{align*}
Since $|G| < {\prod_i|G_i|}$ the lemma follows.
\end{proof}

Indeed the proof of this lemma tells us that
$\rho(n)/\mathcal N_G  \sim |Q|/|[G,G]|$.

Before we proceed, it is clear from the diagrams 
(\ref{eqn: first diagram},\ref{eqn: commutative diagram}) 
that if $G/[G,G]=1$, then the induced
map on homology $H_1(E(2,G)) \to H_1(B(2,G))$ is onto (without using the 5-term sequence
in homology).

\begin{prop}
{If $G$ is simple, then the following map is a surjection
$$H_1(E(2,G);\Z) \to H_1(B(2,G);\Z). $$ }
\end{prop}

Instead, using only topology we want to prove the following equivalent statements: 
\textit{if the map $H_1(E(2,G)) \to H_1(B(2,G))$ is onto, then $|G|$ is even,} 
or equivalently,
 \textit{if $|G|$ is odd, then the map $H_1(E(2,G)) \to H_1(B(2,G))$ is not onto}.

\

Now, if $|G|$ is odd, then all $Q, \, \prod_i G_i$ and $[G,G]$ are odd. Also $N$ is a free group 
of odd index in both free groups $\pi_1(E(2,G))$ and $\pi_1(EG_i,G_i)^{K_0}$.
The following results are immediate:
\begin{lemma}
Either all $\rho(n),\, \mathcal N_G,\, {\rm rank}(N)$ are even, 
or, all $\rho(n),\,\mathcal N_G,\, {\rm rank}(N)$ are odd, 
such that the ratios
$$
\frac{{\rm rank}(N)-1}{\mathcal N_G-1},\, \frac{{\rm rank}(N)-1}{\rho(n)-1}, \, \frac{\rho(n)-1}{\mathcal N_G-1}
$$
are odd.
\end{lemma}
\begin{proof}
Use the formulas in the proof of the previous Lemma.
\end{proof}

Next note that the map $(i_1)_\ast$ in (\ref{eqn: small commutative diagram})
can be a surjection only if $Q \lneq {\rm Im}((i_1)_\ast)$ 
(if not then their intersection is at most $Q$ and the image 
cannot be everything). Consider the following diagram
\begin{center}
\begin{tikzcd}
    & {\rm Ker}_1 \arrow{d}  & \\
{\rm Ker}_2 \arrow{r}  & H_1(E(2,G)) \arrow{d}{q_1} 
		\arrow{r}{(i_1)_\ast} & {\rm Im}((i_1)_\ast) \arrow{d}\\
 & Q \arrow{r}{i} 	&   {\rm Im}((i_1)_\ast).
\end{tikzcd}
\end{center}
The image ${\rm Im}((i_1)_\ast) $ has odd order. Since both kernels have full rank
(the quotients are both finite groups)
we have that ${\rm Ker}_2 \leq {\rm Ker}_1$. The kernels have bases as follows
$$ 	{\rm Ker}_1=span\{\alpha_i e_i: i \in[N_G]\},		\text{ and } 
	{\rm Ker}_2=span\{\beta_i e_i: \beta_i|\alpha_i,\, \forall i \in[N_G]\}.$$
Here all $\alpha_i,\beta_i$ have to be odd numbers such that $\alpha_i|\beta_i $ for all $i$. 
Indeed this can be done for any (finite) sequence of subgroups
$$ Q < Q_1 < \cdots < {\rm Im}((i_1)_\ast) < \cdots <\prod_i G_i$$
as there are kernels $K_0,\, K_1,\dots,$ of full ranks corresponding to projections.
The following theorem shows that ${\rm Im}((i_1)_\ast) \lneq \prod_i G_i$ for the
case of CT groups with trivial center.
We conclude this section with the following corollary.

\begin{cor}\label{cor: top. equiv. form CT groups}
Finite CT groups with trivial center are solvable if and only if the induced map
$H_1(E(2,G);\Z) \to H_1(B(2,G);\Z)$ is not a surjection.
\end{cor}

Of course this is a special case of the condition in \cite[Proposition 7.2]{fredb2g},
but for this corollary we use only the diagrams (\ref{eqn: commutative diagram},\ref{eqn: small commutative diagram}).

\begin{proof}[Proof of Corollary \ref{cor: top. equiv. form CT groups}]
Consider the commutative diagram (\ref{eqn: commutative diagram}) and 
(\ref{eqn: small commutative diagram}). If the map $(i_1)_\ast$ is a 
surjection, then the composition
$$H_1(E(2,G);\Z) \twoheadrightarrow H_1(B(2,G);\Z) \twoheadrightarrow G/[G,G]$$
is a surjection. On the other hand, from diagram (\ref{eqn: commutative diagram}),
the group $H_1(E(2,G);\Z)$ maps trivially, hence $G/[G,G]$ is trivial. 
On the other hand, if $G/[G,G]$ is trivial, then $Q=\prod_i G_i$ and $(i_1)_\ast$
is a surjection.
\end{proof}

The following question is still open for CT groups: 
{\it Use Corollary \ref{cor: top. equiv. form CT groups} to show that if
$G$ is a simple finite TC group with trivial center, then $G$ has even order}.


\end{document}